\def\preprint{}
\ifx\preprint\undefined
\documentclass[12pt]{elsarticle}
\journal{Linear Algebra and its Applications}
\else
\documentclass[11pt,a4paper]{scrartcl}
\fi

\usepackage[utf8]{inputenc}
\usepackage[english]{babel}

\usepackage[%
    colorlinks=true,
    linkcolor=red,
    filecolor=green,
    citecolor=red,
    pdfpagemode=None]{hyperref}

\usepackage{amsmath}
\usepackage{amsfonts}
\usepackage{amsthm}
\usepackage{adjustbox}

\usepackage{graphicx}
\graphicspath{{figures/}{data/}{plots/}{inner-tol-test/}}

\usepackage{algorithm}
\usepackage{algorithmicx}
\usepackage{algpseudocode}
\usepackage{enumitem}
\usepackage{booktabs}

\newtheorem{theorem}{Theorem}[section]
\newtheorem{lemma}[theorem]{Lemma}
\newtheorem{proposition}[theorem]{Proposition}
\newtheorem{corollary}[theorem]{Corollary}

\newtheorem{definition}[theorem]{Definition}
\newtheorem{remark}[theorem]{Remark}

\usepackage{xcolor}
\usepackage{tikz}

\ifx\preprint\undefined
\else
    \newenvironment{frontmatter}{}{}
    \newenvironment{keyword}{\begin{quote}\minisec{Keywords}}{\end{quote}}
    \newcommand{\sep}{, }
    \newcommand{\MSC}[1][]{#1 MSC: }

    \usepackage{helvet}
    \usepackage{newtxtext}
    \usepackage[bigdelims]{newtxmath}

    \usepackage[url=false,doi=false,isbn=false,eprint=false,sorting=nyt,backend=bibtex]{biblatex}
\fi

\newcommand{\dotp}[2]{\langle #1, #2 \rangle}
\newcommand{\norm}[2][2]{\| #2 \|_{#1}}
\newcommand{\Anorm}[1]{\norm[A]{#1}}
\newcommand{\vs}[2]{ {#1}^{#2} } % vector space
\newcommand{\Kvs}[1]{ \vs{\mathbb{K}}{#1} } % K vector space
\newcommand{\mvs}[3]{ {#1}^{#2 \times #3}} % Matrix vector space
\newcommand{\Rmvs}[2]{ \mvs{\mathbb{R}}{#1}{#2} }

\newcommand{\Kmvs}[2]{ \mvs{\mathbb{K}}{#1}{#2} }
\newcommand{\Krylov}[1]{ \mathcal{K}_{#1} }
\newcommand{\Span}{ \mathop{\mathrm{span}} }
\newcommand{\range}{ \mathop{\mathrm{range}} }

\newcommand{\conj}[1]{\overline{#1}}
\newcommand{\dist}{ \mathrm{dist} }

\newcommand{\diag}{ \mathrm{diag} }

\newcommand{\eff}{\mathrm{eff}}

\newlength{\equalssignwidth}
\newcommand{\scinum}[2]{\ensuremath{#1 \cdot 10^{#2}}}

\begin{document}

\begin{frontmatter}

\title{The Deflated Conjugate Gradient Method: %
    Convergence, Perturbation and Accuracy
}
%\titlerunning{Deflated CG: Convergence, Perturbation and Accuracy}

\ifx\preprint\undefined
\author[wtal]{K.~Kahl}
\ead{kkahl@math.uni-wuppertal.de}

\author[wtal]{H.~Rittich}
\ead{rittich@math.uni-wuppertal.de}

\address[wtal]{
    Fakult\"at für Mathematik und Naturwissenschaften,
    Bergische Universit\"at Wuppertal, 42097 Wuppertal,
    Germany
}
\else
    {
    \renewcommand{\thefootnote}{\fnsymbol{footnote}}
    \author{K.~Kahl\footnotemark[1], H.~Rittich\footnotemark[1]}
    \date{October 17, 2016}
    \maketitle
    \footnotetext[1]{
        Fakult\"at für Mathematik und Naturwissenschaften,
        Bergische Universit\"at Wuppertal, 42097 Wuppertal,
        Germany \\
        e-mail: kkahl@math.uni-wuppertal.de (K.~Kahl), 
            rittich@math.uni-wuppertal.de (H.~Rittich)
    }
    }
\fi

\begin{abstract}%
  Deflation techniques for Krylov subspace methods have seen a lot of
  attention in recent years. They provide means to improve the convergence
  speed of these methods by enriching the Krylov subspace with a deflation
  subspace. 
  The most common approach for the construction of deflation
  subspaces is to use (approximate) eigenvectors, but also more general
  subspaces are applicable.

  In this paper we discuss two results concerning the accuracy requirements
  within the deflated CG method.  First we show that the effective condition
  number which bounds the convergence rate of the deflated conjugate gradient
  method depends asymptotically linearly on the size of the perturbations in
  the deflation subspace.  Second, we discuss the accuracy required in
  calculating the deflating projection.  This is crucial concerning the
  overall convergence of the method, and also allows to save some
  computational work.

  To show these results, we use the fact that as a projection approach
  deflation has many similarities to multigrid methods.  In particular, recent
  results relate the spectra of the deflated matrix to the spectra of the
  error propagator of twogrid methods.  In the spirit of these results we show
  that the effective condition number can be bounded by the constant of a weak
  approximation property.
\end{abstract}


\end{frontmatter}

%% \linenumbers

\section{Introduction}
Consider solving the linear system of equations
\begin{equation} \label{basic_linear_system:eq}
    A x = b
    \,,
\end{equation}
where $A \in \Kmvs{n}{n}$ ($\mathbb{K} = \mathbb{R}$ or $\mathbb{K} =
\mathbb{C}$) is self-adjoint and positive definite and $x, b \in \Kvs{n}$. In
this paper we are interested in the case where the matrix $A$ is
large and sparse. The conjugate gradient (CG) method
\cite{GLMatrix1989,HSMethods1952,SaaIterative2003} is an iterative method which is often well
suited to solve these systems. 
The speed of convergence of the CG method depends
on the distribution of its eigenvalues and the right hand side.
Estimates for the speed of convergence in terms of the condition of the matrix
$A$ exist \cite{SaaIterative2003,SVRate1986}. 
When the condition number $\kappa$ is large it can become mandatory to
precondition the linear system such that a satisfactory speed of convergence
can be guaranteed.

One possibility to precondition the CG method is via deflation as
introduced by Nicolaides~\cite{NicDeflation1987} and Dostal~\cite{DosConjugate1988} (see also
\cite{AMNWDeflated2010,CWAnalysis1997,FVConstruction2001,GGLNFramework2013,LuesLocal2007,SYEGDeflated2000,SPMConjugate1989}).
The basic idea of deflation is to ``hide''
certain parts of the spectrum of the matrix $A$ from the CG method itself,
such that the CG iteration ``sees'' a system that has a much smaller
condition number than $A$.
The part of the spectrum that is hidden from CG is determined by the
{\em deflation subspace} $\mathcal{S} \subseteq \Kvs{n}$
and the improvement of the convergence rate of the deflated CG method hinges
solely on the choice of $\mathcal{S}$.

One viable and widely used approach for deflation consists of spanning
$\mathcal{S}$ by the eigenvectors corresponding to the smallest eigenvalues
\cite{SYEGDeflated2000}.  This then hides the smallest eigenvalues from the
spectrum of $A$. On the other hand, already in \cite{NicDeflation1987} a
different choice of $\mathcal{S}$ has been used and further examples for other
deflation subspaces can be found in \cite{FVConstruction2001,LuesLocal2007}.

Different two-level
approaches, including deflation and twogrid methods 
\cite{BHMMultigrid2000,HacMulti1985,TOSMultigrid2001},
have been compared in \cite{NVComparison2006,TNVEComparison2009} and equivalences 
between these apparently different methods have been established.
In continuation of these efforts to connect different two-level approaches, we show how to 
use theory developed for multigrid methods to
analyze the deflated CG method.

After giving a short introduction to the deflated CG method in
Section~\ref{sec:deflation} 
we give an overview on how to apply multigrid theory in the deflation context
in Section~\ref{sec:convergence-analysis}.

Based on these theoretical considerations we answer two different questions
concerning accuracy that arise in the deflation approach.

First, in Section~\ref{sec:inexact-deflation} 
we analyze the situation where the deflation subspace
$\mathcal{S}$ is only known up to a perturbation and give a
bound on the effective condition number with respect to the size of
the perturbation. 
This result is of particular interest in the case where the deflation 
subspace is spanned by
eigenvectors corresponding to the smallest eigenvalues of $A$.
In many practical applications eigenvectors are unknown and need to be
approximated and thus the resulting deflation subspace can be thought
of as a perturbation of the deflation subspace that uses the exact eigenvectors.
This result can in particular be applied to the analysis of methods like
eigCG \cite{SOComputing2010} and similar methods 
\cite{AMNWDeflated2010,CWAnalysis1997,SPMConjugate1989}
where the eigenvectors are approximated numerically.

Second, we consider the situation when the dimension of $\mathcal{S}$ is
large.  In this case special care is needed within a deflation type method as
the deflating projection now involves the solution of 
a large linear system with the matrix $V^*AV$, where the columns of $V$ form a
basis of $\mathcal{S}$.
We discuss the accuracy requirements on the solution of this system
to ensure proper convergence of the deflated CG method based on results from
\cite{SSTheory2003,ESInexact2004} on inexact Krylov subspace methods.
This reveals a way to reduce computational
work needed to perform the deflated CG method.

Some numerical experiments confirming and illustrating the theory are given 
in Section~\ref{sec:experiments}.

We conclude this introduction by fixing basic notation used throughout the paper. 
For $z \in \Kvs{n}$ its residual $r \in \Kvs{n}$ is given by
$ r = b - A z$, the error by $ e = x - z$ where
$x$ is the solution of \eqref{basic_linear_system:eq}.
Note that 
$A e = r$.

Let $\dotp{v}{w} = \sum_{i = 1}^n \conj{w_i}\, v_i$ be the euclidean inner
product of $v$ and $w$, $\norm{v} = \dotp{v}{v}^{1/2}$ the euclidean norm of
$v$.
Since $A$ is self-adjoint and positive definite the $A$-inner
product and the $A$-norm exist and are given by
\[
    \dotp{v}{w}_A := \dotp{A v}{w}
    \quad \text{and} \quad
    \Anorm{v} := \dotp{v}{v}_A^{1/2}
    \,.
\]
Let $q_1, q_2, \dots, q_n$ be an orthonormal basis of eigenvectors
of the matrix $A$, s.t.\ $\lambda_1 \ge \lambda_2 \ge \dots \ge \lambda_n$ are
the corresponding eigenvalues.

\section{Review of Deflated CG}\label{sec:deflation}

Let $x_0$ be an initial guess,
$r_0 = b - A x_0$ and the $i$th Krylov subspace be denoted by
$\Krylov{i}(A, r_0) := \Span\{ r_0, A r_0, A^2 r_0, \dots, A^{i-1} r_0\}$.
The $i$-th CG iterate is determined such that $x_i \in x_0 + \Krylov{i}(A, r_0)$
and the error $e_i = x - x_i$ is minimized in the $A$-norm
(cf.~\cite{SaaIterative2003}), i.e.,
\begin{equation}
    \label{eq:a-norm-error-cg}
    \Anorm{ e_i } 
    = \min \{ \Anorm{ x - z } : z \in x_0 + \Krylov{i}(A, r_0) \}
    \,.
\end{equation}
Note that \eqref{eq:a-norm-error-cg} is just the distance in the $A$-norm
between $x$ and the affine subspace $x_0 + \Krylov{i}(A, r_0)$.  The convergence
of the method can be slow in case of unfavorable spectral properties of the
matrix $A$ \cite{SVRate1986}. 
The idea of deflation is to modify the CG method such that
the iterates $x_i$ are equal to the solution $x$ on a given subspace
$\mathcal{S} \subseteq \Kvs{n}$ in the sense that
the $A$-orthogonal projections of $x_i$
and $x$ onto $\mathcal{S}$ coincide (and are thus identical for all $i$).
For proper choices of $\mathcal{S}$ this will improve the speed of
convergence since the affine subspace which contains the iterates may now be
much closer to the solution $x$ than the original one. We give a rigorous
description of the deflated CG method in the remainder of this section.

Let $\mathcal{S}^{\perp_A} = (A \mathcal{S})^\perp$ 
be the $A$-orthogonal complement of a
given subspace 
$\mathcal{S} \subseteq \Kvs{n}$.
We can split the solution $x$ into a
component in $\mathcal{S}$ and a component in $\mathcal{S}^{\perp_A}$
via the $A$-orthogonal projection $\pi_A(\mathcal{S}) \in \Kmvs{n}{n}$ onto
$\mathcal{S}$, i.e.,
\begin{equation}
    \label{eq:splitting-of-x}
    x = (I - \pi_A(\mathcal{S})) x + \pi_A(\mathcal{S}) x
    \,.
\end{equation}
Let $V \in \Kmvs{n}{m}$ be a matrix such that its columns form a basis of the
subspace $\mathcal{S}$. Since
\begin{equation}
    \label{eq:solution-in-S}
    \pi_A(\mathcal{S}) x
    = V (V^* A V)^{-1} V^* A x
    = V (V^* A V)^{-1} V^* b
\end{equation}
we can compute $\pi_A(\mathcal{S}) x$---the second term in the right hand side
of \eqref{eq:splitting-of-x}---without explicit knowledge of $x$.
The first term of \eqref{eq:splitting-of-x} can be computed from a solution
$\hat x$ of the singular linear system
\begin{equation}
    \label{eq:deflated_system}
    A (I - \pi_A(\mathcal{S})) \hat x = (I - \pi_A(\mathcal{S}))^* b
    \,,
\end{equation}
which we call the {\em deflated (linear) system}. 
For the sake of completeness we prove this in the following lemma.
Its statements can be found, e.g., in \cite{FVConstruction2001,LuesLocal2007}.
\begin{lemma}%
    \label{lem:defl-sys-prop}%
    Using the definitions from above we have:
    \begin{enumerate}[leftmargin=3em,label=(\roman*)]
        \item \label{itm:symmetry_relation}
        The following equalities hold
        \begin{equation}
            \label{eq:symmetry_relation}
            A (I - \pi_A(\mathcal{S}))
            = (I - \pi_A(\mathcal{S}))^* A 
            = (I - \pi_A(\mathcal{S}))^* A (I - \pi_A(\mathcal{S}))
            \,.
        \end{equation}
        \item \label{itm:spd_deflated_system}
        The matrix $A (I - \pi_A(\mathcal{S}))$ is self-adjoint and
        positive semi-definite.
        \item \label{itm:consistent_system}            
        The deflated system \eqref{eq:deflated_system} is consistent, 
        i.e., the right hand side 
        $(I - \pi_A(\mathcal{S}))^* b$ is in the range of
        $A (I - \pi_A(\mathcal{S}))$.
        This implies that the system has at least one solution.
        \item \label{itm:solve-via-deflated-system}
        If $\hat x$ is a solution of the deflated system
        \eqref{eq:deflated_system} then
        \begin{equation}
            \label{eq:solve-via-deflated-system}
            (I - \pi_A(\mathcal{S})) \hat x = (I - \pi_A(\mathcal{S})) x
            \,,
        \end{equation}
        where $x$ is the solution of the linear system $Ax = b$.
    \end{enumerate}
\end{lemma}
\begin{proof} 
    Since $\pi_A(\mathcal{S}) = V (V^* A V)^{-1} V^* A$ we have
    \[
        A (I - \pi_A(\mathcal{S}))
        = A - A V (V^* A V)^{-1} V^* A
        = (I - \pi_A(\mathcal{S}))^* A
        \,,
    \]
    and since $I - \pi_A(\mathcal{S})$ is a projection we also get
    \[
        A (I - \pi_A(\mathcal{S})) 
        = A (I - \pi_A(\mathcal{S})) (I - \pi_A(\mathcal{S})) 
        = (I - \pi_A(\mathcal{S}))^* A (I - \pi_A(\mathcal{S}))
        \,.
    \]
    This proves \ref{itm:symmetry_relation} and also shows that 
    $A (I - \pi_A(\mathcal{S}))$ is self-adjoint.
    Using \eqref{eq:symmetry_relation} and due to
    $\dotp{B^* A B x}{x} = \dotp{A B x}{B x} = \dotp{Bx}{Bx}_A$,
    $B \in \Kmvs{n}{n}$ we have
    \[
        \dotp{A (I - \pi_A(\mathcal{S})) x}{x}
        = \dotp{(I - \pi_A(\mathcal{S})) x}{(I - \pi_A(\mathcal{S})) x}_A
        = \Anorm{ (I - \pi_A(\mathcal{S})) x }^2
        \ge 0
    \]
    which gives \ref{itm:spd_deflated_system}.

    Again due to \eqref{eq:symmetry_relation} and the fact that $A$ has full
    rank we have
    \[
        \range \big( A (I - \pi_A(\mathcal{S})) \big)
        = \range \big( (I - \pi_A(\mathcal{S}))^* A \big)
        = \range \big( (I - \pi_A(\mathcal{S}))^* \big)
        \,.
    \]
    Hence the system \eqref{eq:deflated_system}
    is consistent which proves \ref{itm:consistent_system}.
    To show \ref{itm:solve-via-deflated-system}
    we use \eqref{eq:deflated_system} and \eqref{eq:symmetry_relation}
    yielding
    \[
        A (I - \pi_A(\mathcal{S})) \hat x
        = (I - \pi_A(\mathcal{S}))^* b
        = (I - \pi_A(\mathcal{S}))^* A x
        = A (I - \pi_A(\mathcal{S})) x
        \,.
    \]
    Multiplying with $A^{-1}$ from the left we obtain
    \eqref{eq:solve-via-deflated-system}.
\qquad\end{proof}

To sum up,
given a solution $\hat x$ for the
deflated system \eqref{eq:deflated_system}
we can compute the first part of the splitting
\eqref{eq:splitting-of-x} by \eqref{eq:solve-via-deflated-system} and the
second part by using the formula \eqref{eq:solution-in-S}.
Therefore we obtain the solution $x$ for the original system as
\begin{align*}
    x &= (I - \pi_A(\mathcal{S})) x + \pi_A(\mathcal{S}) x \\
      &= (I - \pi_A(\mathcal{S})) \hat x + V (V^* A V)^{-1} V^* b
    \,.
\end{align*}
This relation allows us to compute approximations to the solution of the
original system by computing an approximation to the solution $\hat x$ of the
deflated system \eqref{eq:deflated_system}.
Since $A (I - \pi_A(\mathcal{S}))$ is positive semi-definite 
we can apply the CG method. The fact that the matrix is singular is no impediment to the
standard CG iteration as long as~\eqref{eq:deflated_system} is consistent
(cf.~\cite{KaaPreconditioned1988}), which has been shown to be the case in
Lemma~\ref{lem:defl-sys-prop}.

For the purpose of analyzing the method we think of deflated
CG as applying the standard CG algorithm 
to the deflated system \eqref{eq:deflated_system}
with the matrix $A (I-\pi_A(\mathcal{S}))$.
There are various other mathematically equivalent formulations of deflated CG
(for an overview see
\cite{GGLNFramework2013}) for which our analysis holds as well.

Let $\mu_1 \ge \dots \ge \mu_n \ge 0$ be the eigenvalues of the self-adjoint and
positive semi-definite matrix $A (I - \pi_A(\mathcal{S}))$.
Let $k \in \mathbb{N}$ denote the largest index such that 
$\mu_k \not= 0$. The errors of the CG iterates then satisfy
\begin{equation}\label{eq:convergence_estimate_dCG}
	\Anorm{ e_i }
	\le 2 \left( \frac{ \sqrt{\kappa_\eff} - 1 }{ \sqrt{\kappa_\eff} + 1} \right)^i
	\Anorm{ e_0 }
	\quad \text{for } i = 0, 1, 2, \ldots,
\end{equation}
where $\kappa_\eff = \tfrac{\mu_1}{\mu_k}$, see 
\cite{FVConstruction2001,SYEGDeflated2000}.
We call $\kappa_\eff$ the effective condition number of the deflated matrix $A
(I - \pi_A(\mathcal{S}))$ to distinguish it from the condition number $\kappa$
of the original matrix $A$.
Thus a bound on the convergence rate of deflated CG
can be obtained by estimating
the largest and smallest {\em non-zero}
eigenvalue of the matrix $A (I - \pi_A(\mathcal{S}))$.

\section{Convergence Analysis}
\label{sec:convergence-analysis}
We estimate the effective condition number $\kappa_\eff$
of the matrix $A (I - \pi_A(\mathcal{S}))$
in terms of a quantity
which arises in the {\em weak approximation property} used in
multigrid theory.

\begin{definition}
\label{def:wap}
A subspace $\mathcal{S} \subseteq \Kvs{n}$ fulfills the 
{\em weak approximation property with constant $K \ge 0$} if
$K$ is the \emph{smallest} number s.t.
\begin{equation}
    \label{eq:weak_approx}
    \norm{ x - \pi(\mathcal{S}) x }^2
    \le \frac{K}{ \norm{A} } \Anorm{ x }^2
    \quad \text{for all } x \in \Kvs{n}
    \,.
\end{equation}
Here, $\pi(\mathcal{S})$ denotes the $\ell_2$-orthogonal projection onto
$\mathcal{S}$.
\end{definition}

Note that $\norm{ x - \pi(\mathcal{S}) x }$ is 
the {\em $\ell_2$-distance between $x$ and the subspace $\mathcal{S}$}
defined by
\begin{equation}
    \label{eq:def-subspace-dist}
    \dist(\mathcal{S}, x)_2
    = \min_{y \in \mathcal{S}} \norm{ x - y }
    \,.
\end{equation}
Up to a scaling by the diagonal entries of $A$
this definition coincides with the definition of the weak approximation
property found in the multigrid literature (see, e.g., 
\cite{BraAlgebraic1986,BCF+Algebraic2001,RSAlgebraic1987,StuIntroduction2001}).
It is called ``weak'' because it is only
sufficient for a two-level convergence theory 
\cite[Section~4.5]{RSAlgebraic1987} but not for a multilevel one.
Note that any subspace $\mathcal{S}$ fulfills a weak approximation 
property, however $K$ may be large.
In order to guarantee fast twogrid convergence one is 
interested in subspaces which admit a small value for $K$.
Uniform bounds for $K$ exist for particular families of matrices and certain
subspaces and can be directly derived from the results in, e.g.,
\cite{BraAlgebraic1986,BCF+Algebraic2001,RSAlgebraic1987,StuIntroduction2001}.  Such
families are typically different levels of discretization of a continuous
operator.

To derive the bound for the effective condition number we
start with the following auxiliary result.
\begin{lemma}
    \label{eq:ew-A-and-orth-projection}
    Let $\{ 0 \} \not= \mathcal{S} \subseteq \Kvs{n}$ be a subspace
    and $k := n - \dim(\mathcal{S})$.
    Furthermore, let $v_1, \dots, v_n$ be a basis of $\Kvs{n}$ consisting of
    eigenvectors corresponding to the eigenvalues
    $\mu_1 \ge \dots \ge \mu_n \ge 0$ of $A (I - \pi_A(\mathcal{S}))$.

    Then the vectors $v_1, \dots, v_k, A v_{k+1}, \dots, A v_n$ form a basis
    of $\Kvs{n}$ consisting of eigenvectors 
    corresponding to the eigenvalues $\mu_1^{-1}, \dots, \mu_k^{-1}, 0, \dots, 0$
    of $(I-\pi(\mathcal{S})) A^{-1}$.
    Thus
    \[
        \sigma( (I-\pi(\mathcal{S})) A^{-1})
        = \{ \mu_1^{-1}, \dots, \mu_k^{-1}, 0 \}
        \,,
    \]
    where $\sigma((I-\pi(\mathcal{S})) A^{-1})$ is the spectrum of
    $(I-\pi(\mathcal{S})) A^{-1}$.
\end{lemma}
\begin{proof}
    See \cite[Theorem~2.1]{NotAlgebraic2010}, 
    cf.~\cite[Theorem 3.24]{GauRecycling2014}.
\end{proof}

As $\sigma(AB) = \sigma(BA)$
for general matrices $A, B \in \Kmvs{n}{n}$ we obtain the following result
from Lemma~\ref{eq:ew-A-and-orth-projection}.

\begin{corollary}
    \label{cor:sym-inv-ew-relation}
    Under the same assumptions as in Lemma~\ref{eq:ew-A-and-orth-projection} we
    have
    \[
        \sigma( A^{-1/2} (I - \pi(\mathcal{S})) A^{-1/2} ) =
        \{ \mu_1^{-1}, \dots, \mu_k^{-1}, 0 \}
        \,.
    \]
\end{corollary}

With Corollary~\ref{cor:sym-inv-ew-relation} we are now able to formulate the
main theorem of this section.

\begin{theorem}
    \label{thm:wap-estimate}
    Let $\mu_1 \ge \mu_2 \ge \dots \ge \mu_n \ge 0$ be the eigenvalues of 
    $A (I - \pi_A(\mathcal{S}))$ and $k$ the largest integer s.t. 
    $\mu_k \not= 0$.
    Furthermore, let the weak approximation property \eqref{eq:weak_approx} hold with constant $K$.
    Then
    \[
        \mu_1 \le \norm{A}
        \quad \text{and} \quad
        \mu_k = \tfrac{ \norm{A} }{ K }
        \,.
    \]
    And consequently the effective condition number $\kappa_\eff$ of the matrix
    $A (I - \pi_A(\mathcal{S}))$ fulfills
    \[
        \kappa_\eff \le K
        \,.
    \]
\end{theorem}
\begin{proof}
    We first show that $\norm{A}$ is an upper bound for $\mu_1$.
    As the matrix $A \pi_A(\mathcal{S}) = A V (V^* A V)^{-1} V^* A$ is
    positive semi-definite and thus
    \[
        \dotp{A (I - \pi_A(\mathcal{S})) x }{x}
        = \dotp{A x}{x} - \dotp{A \pi_A(\mathcal{S}) x}{x}
        \le \dotp{A x}{x}
        \,,
    \]
    we obtain, characterizing eigenvalues by Rayleigh quotients
    (see, e.g., \cite{WilAlgebraic1965}),
    \begin{equation}
    \label{eq:mu_upper_bound}
        \mu_1
        = \max_{v \in \Kvs{n} \setminus \{0\}}
        \frac{\dotp{A (I - \pi_A(\mathcal{S})) v}{v}}{%
            \dotp{v}{v}} 
        \le \max_{x \in \Kvs{n} \setminus \{0\} }
            \frac{ \dotp{A x}{x} }{ \dotp{x}{x} }
        = \lambda_1 = \norm{A}
        \,,
    \end{equation}
    where $\lambda_1$ was defined as the largest eigenvalue of $A$.

    We now prove that $\tfrac{ \norm{A} }{K}$ is equal to $\mu_k$. 
    From
    Corollary~\ref{cor:sym-inv-ew-relation} we see that
    \[
        \mu_k^{-1} = \max_{v \not= 0}
        \frac{\dotp{A^{-\frac 12}(I - \pi(\mathcal{S})) A^{-\frac 12}v}{v}}{%
            \dotp{v}{v}}
        = \max_{v \not= 0}
        \frac{\dotp{(I - \pi(\mathcal{S})) A^{-\frac 12}v}{A^{-\frac 12} v}}{%
            \dotp{v}{v}}
        \,.
    \]
    Substituting $v$ by $A^{\frac12} w$ and using $(I -
    \pi(\mathcal{S}))^2 = (I - \pi(\mathcal{S}))
    = (I - \pi(\mathcal{S}))^*$ yields
    \[
        \mu_k^{-1}
        = \max_{v \not= 0}
        \frac{\dotp{(I - \pi(\mathcal{S})) w}{w}}{%
        \dotp{A^{\frac12} w}{A^{\frac12} w}}
        = \max_{v \not= 0}
        \frac{\dotp{(I - \pi(\mathcal{S})) w}{(I - \pi(\mathcal{S}))w}}{%
        \dotp{A w}{w}}
        \,.
    \]
    And by using the definition of the weak approximation property \eqref{eq:weak_approx} 
    we have that $\mu_k^{-1} = K / \norm{A}$.
\qquad\end{proof}

\begin{remark}
    \label{rem:previous-work}
    \textrm{
    Theorem~\ref{thm:wap-estimate} can also be proven by
    using the spectral equivalence of the deflated CG method and
    the CG method preconditioned by a $V(1,0)$-cycle of a multigrid method    
    with a Richardson smoother 
    with weight equal to one
    \cite[Theorem~3.3]{TNVEComparison2009}.
    Then Section~1 and Theorem~2.1 in \cite{NNFurther2013} after some
    algebraic simplifications, yield the above theorem.}
\end{remark}

\section{Perturbation of Deflation Subspaces}
\label{sec:inexact-deflation}

We now use Theorem~\ref{thm:wap-estimate} to give a bound
on the effective condition number when the deflation subspace is
perturbed.
This is of particular interest when the deflation subspace is spanned by
eigenvectors of the matrix $A$ which is common practice, see e.g.
\cite{AMNWDeflated2010,CWAnalysis1997,FVConstruction2001,GGLNFramework2013,%
SYEGDeflated2000,SPMConjugate1989,SOComputing2010}.
Typically the eigenvectors corresponding to the smallest
eigenvalues are unknown and need to be approximated numerically.  Thus the
question arises how precisely those eigenvectors need to be determined to
achieve fast convergence of the deflated CG method.
We first answer this question for the perturbation of general deflation
subspaces and later on discuss the case of deflation of eigenvectors. 

We need to quantify the perturbation of the deflation subspace to give a bound
on the effective condition number.
The difference between two subspaces $\mathcal{S}$ and
$\widetilde{\mathcal{S}}$ can be measured in terms of the largest principle
angle $\theta$---also called \emph{subspace gap}---between the subspaces
(cf.~\cite{GLMatrix1989,SSOccurrence2005}), i.e.,
\begin{equation}
    \label{eq:subspace-gap}
    \| \pi(\mathcal{S}) - \pi(\widetilde{\mathcal{S}}) \|_2
    = \sqrt{1 - \cos(\theta)^2} = \sin(\theta).
\end{equation}
Thus we can measure the perturbation as the angle between the perturbed and
the unperturbed subspace. This is a suitable measure as it does not
depend on the choice of the basis of the subspace.
In the following Lemma we give a bound on the weak approximation property
constant using the subspace angle $\theta$.

\begin{lemma}
    \label{lem:perturb-angles}
    Let $\mathcal{S} \subseteq \Kvs{n}$ and $\widetilde{\mathcal{S}} \subseteq
    \Kvs{n}$ be two subspaces of the same dimension. Assume that the space
    $\mathcal{S}$ fulfills the weak approximation property with constant $K$.
    Then $\widetilde{\mathcal{S}}$ fulfills the weak approximation property
    with constant
    \[
        \widetilde{K}
        \le \left( K^{1/2} + \sin(\theta) \cdot \sqrt{\kappa(A)} \right)^2
        \,,
    \]
    where $\theta$ is the largest principal angle between $\mathcal{S}$ and
    $\widetilde{\mathcal{S}}$.
\end{lemma}
\begin{proof}
    Let $x \in \Kvs{n}$. Then
    \begin{align}
            \notag
            \| x - \pi(\widetilde{\mathcal{S}}) x \|_2
        &=
            \| x + \pi(\mathcal{S})x - \pi(\mathcal{S})x 
                - \pi(\widetilde{\mathcal{S}})x \|_2 \\
        &\le
            \label{eq:subspace-proof-1}
            \| x - \pi(\mathcal{S})x \|_2 
            +   \| \pi(\mathcal{S})x - \pi(\widetilde{\mathcal{S}})x \|_2
        \,.
    \end{align}
    The subspace $\mathcal{S}$ fulfills the weak approximation property thus
    the first term of \eqref{eq:subspace-proof-1} is bound by
    \[
        \| x - \pi(\mathcal{S}) \|_2 
        \le K^{1/2} \cdot \tfrac{ \| x \|_A}{\| A \|_2^{1/2}}
        \,.
    \]
    We can bound the second term \eqref{eq:subspace-proof-1} by using the
    subspace gap \eqref{eq:subspace-gap} to get
    \begin{align*}
            \| \pi(\mathcal{S})x - \pi(\widetilde{\mathcal{S}})x \|_2
        &\le 
            \| \pi(\mathcal{S}) - \pi(\widetilde{\mathcal{S}}) \|_2 \cdot
            \| x \|_2 
        =
            \sin(\theta) \cdot \| x \|_2 \\
        &\le
            \sin(\theta) \cdot \tfrac{ \| x \|_A }{\lambda_n^{1/2}}
        =
            \sin(\theta) \cdot \tfrac{\lambda_1^{1/2}}{\lambda_n^{1/2}} 
            \cdot \tfrac{ \| x \|_A }{\| A \|_2^{1/2}}
        \,.
    \end{align*}
    Combining the two estimates yields
    \[
            \| x - \pi(\widetilde{\mathcal{S}}) x \|_2
        \le 
            \left(
                K^{1/2} + \sin(\theta) \cdot
                \tfrac{\lambda_1^{1/2}}{\lambda_n^{1/2}}
            \right)
            \cdot \tfrac{ \| x \|_A }{\| A \|_2^{1/2}}
        \,.
        \qedhere
    \]
\end{proof}

Using Theorem~\ref{thm:wap-estimate} we get a bound for the smallest
eigenvalue of the deflated matrix:

\begin{theorem}
    \label{thm:mu-min-inv-est}
    Let $\mathcal{S} \subseteq \Kvs{n}$ and 
    $\widetilde{\mathcal{S}} \subseteq \Kvs{n}$ be two subspaces of the same
    dimension. Let $\mu_k$ be the smallest non-zero eigenvalue of
    $A(I - \pi_A(\mathcal{S}))$ and
    $\tilde \mu_k$ the smallest non-zero eigenvalue of
    $A(I - \pi_A(\widetilde{\mathcal{S}}))$. Then
    \begin{equation}
        \label{eq:mu-min-inv-est}
        \tilde \mu_k^{-1} \le \left( \mu_k^{-1/2} + 
        \lambda_n^{-1/2} \cdot \sin(\theta)
        \right)^2
        =
        \mu_k^{-1} + \mathcal{O}(\theta), \quad
        \text{for $\theta \to 0$}
    \end{equation}
    where $\theta$ is the largest principal angle between $\mathcal{S}$ and
    $\widetilde{\mathcal{S}}$.
\end{theorem}

Recall that Theorem~\ref{thm:wap-estimate} stated that the largest eigenvalue
of the deflated matrix is smaller than the largest eigenvalue of the original
matrix $A$. Thus in combination with Theorem~\ref{thm:mu-min-inv-est} this
gives an estimate for the effective condition number.

We now turn our attention to the special case where $\tilde{\mathcal{S}}$ is spanned by
inexact eigenvectors.
Recall that we defined $q_1, q_2, \dots, q_n$ to be an orthonormal basis of eigenvectors
of the matrix $A$, s.t.\ $\lambda_1 \ge \lambda_2 \ge \dots \ge \lambda_n$ are
the corresponding eigenvalues.
Furthermore, 
let $\mathcal{S}$ be the space spanned by the eigenvectors corresponding to
the $(n-k)$ smallest eigenvalues, i.e.,
$\mathcal{S} = \range V$, where
\[
    V = [ q_{k+1} | q_{k+2} | \dots | q_n ]
    \,.
\]

To apply Theorem~\ref{thm:mu-min-inv-est}
we now determine the value for $K$, such that the weak approximation
property \eqref{eq:weak_approx} for $\mathcal{S}$ is fulfilled.
Let $x \in \Kvs{n}$ be given
and $ x = \sum_{i=1}^n \xi_i q_i $
its expansion
in terms of the orthonormal eigenvectors $q_i$ of $A$.
Then the orthogonal projection $\pi(\mathcal{S}) x$ of $x$ onto $\mathcal{S}$
fulfills $\pi(\mathcal{S}) x = \sum_{i=k+1}^n \xi_i q_i$ and thus
\[
    \norm{ x - \pi(\mathcal{S}) x }^2
    = \norm{ \sum_{i=1}^k \xi_i q_i + \sum_{i=k+1}^n (\xi_i - \xi_i) q_i }^2
    = \sum_{i=1}^k |\xi_i|^2
    \,.
\]
This yields
\begin{equation} \label{eq:exact-ev-wap}
    \Anorm{x}^2 
    = \sum_{i=1}^n |\xi_i|^2 \lambda_i
    \ge \sum_{i=1}^k | \xi_i |^2 \lambda_i
    \ge \lambda_k \sum_{i=1}^k |\xi_i|^2
    = \lambda_k \, \norm{x - \pi(\mathcal{S}) x}^2
    \,.
\end{equation}
Hence the weak approximation property \eqref{eq:weak_approx} holds
with $K \le \tfrac{\norm{A}}{\lambda_k} = \tfrac{\lambda_1}{\lambda_k}$.
This is also the smallest possible constant that fulfills the weak
approximation property as \eqref{eq:exact-ev-wap} is an equality for
$x = q_k$. Thus $K = \tfrac{\lambda_1}{\lambda_k}$.
Using Theorem~\ref{thm:wap-estimate}
we obtain $\kappa_\eff \le \tfrac{\lambda_1}{\lambda_k}$.
Furthermore it is known that
$\kappa_\eff = \tfrac{\lambda_1}{\lambda_k}$
(see, e.g.~\cite[Section~1]{FVConstruction2001})
hence our bound is sharp.

Now we are in the position to formulate a new result
concerning the deviation of the effective condition number from
$\tfrac{\lambda_1}{\lambda_k}$ due to inexact eigenvectors.
We apply Theorem~\ref{thm:mu-min-inv-est} and Theorem~\ref{thm:wap-estimate}
to obtain the following proposition.

\begin{proposition}
    \label{prop:eigenvalue-perturbation}
    Let $q_1, \dots, q_n$ be an orthonormal basis of eigenvectors
    corresponding to the eigenvalues
    $\lambda_1 \ge \dots \ge \lambda_n \ge 0$
    of $A$ and let the subspace
    $\mathcal{S} := \Span\{ q_{k+1}, \dots, q_n \}$.
    Then the effective condition number of the deflated system with respect to
    the deflation subspace $\widetilde{\mathcal{S}}$ fulfills
    \begin{equation}
        \label{eq:kappa_eff_perturbed}
        \kappa_\eff 
        \le
        \left( 
            \sqrt{\tfrac{\lambda_1}{\lambda_k}} +
            \sqrt{\tfrac{\lambda_1}{\lambda_n}} \cdot
            \sin(\theta)
        \right)^2
        = \tfrac{\lambda_1}{\lambda_k}
        + \mathcal{O}(\theta)
        \,,
        \quad \text{for $\theta \to 0$}
        \,,
    \end{equation}
    where $\theta$ is the largest principal angle between
    $\mathcal{S}$ and $\widetilde{\mathcal{S}}$.
\end{proposition}

Proposition~\ref{prop:eigenvalue-perturbation} shows that the effective
condition number deviates asymptotically linearly from the unperturbed
effective condition number.  In addition it indicates that the accuracy of the
deflated eigenvectors should increase linearly with the condition number if we
aim at keeping the effective condition number within a given factor of its
optimal value $\tfrac{\lambda_1}{\lambda_k}$.

\begin{remark}
    For some applications the largest principal angle is not a natural measure
    for the perturbation of the deflation subspace.
    Under the following assumption there is a simple bound for the largest
    principal angle.
    Let $V \in \Kmvs{n}{n-k}$ with orthonormal columns and $\mathcal{S} :=
    \range V$.
    Furthermore, let $E \in \Kmvs{n}{n-k}$ and
    $\widetilde{\mathcal{S}} := \range(V+E)$.
    Then the largest principal angle $\theta$ between
    $\mathcal{S}$ and $\widetilde{\mathcal{S}}$ fulfills
    \[
        \sin(\theta) \le \| E \|_2
        \,.
    \]
    In other words if the deflation subspace is given by an orthonormal basis
    and we can bound the norm of the difference of this orthonormal basis and
    a basis of the perturbed deflation subspace then we can bound $\theta$.
\end{remark}
\begin{proof}
    We have \cite[Theorem 5.5]{SSMatrix1990}, 
    \cite[I.~Theorem~6.34]{KatPerturbation1976} that
    \begin{align*}
            \sin(\theta)
        &=  \| \pi(\mathcal{S}) - \pi(\widetilde{\mathcal{S}}) \|_2
        =   \| (I - \pi(\widetilde{\mathcal{S}})) \pi(\mathcal{S}) \|_2
        \,.
        \intertext{It then follows that}
            \sin(\theta)
        &=  
            \max_{\substack{\| x \| = 1 \\ x \in \mathcal{S}}}
            \min_{y \in \widetilde{\mathcal{S}}}
            \| x - y \|_2
        \le
            \max_{\substack{\| x \| = 1 \\ x \in \mathcal{S}}}
            \| x - (V+E)V^* x \|_2
        \,.
        \\
        \intertext{Every $x \in \mathcal{S}$ with $\| x \|$ can be written as
        $x = V z$ with $\| z \| = 1$. Thus}
            \sin(\theta)
        &\le
            \max_{\| z \| = 1}
            \| V z - (V+E)V^* V z \|_2
        =
            \max_{\| z \| = 1}
            \| V z - Vz + Ez \|_2
        \\
        &=  \max_{\| z \| = 1}
            \| Ez \|_2
        = \| E \|_2
        \,.
        \qedhere
    \end{align*}
\end{proof}

\section{Accuracy of the Deflating Projection}
\label{sec:accuracy-discussion}

The deflated CG method involves the solution of the \emph{inner linear system}
\begin{equation}
    \label{eq:inner-system}
    (V^* A V) z_{i+1} = V^* A r_{i+1}
\end{equation}
in every iteration. In the situation where $\mathcal{S}$ is of large
dimension, it can be desirable
to solve the inner system \eqref{eq:inner-system} inexactly by an iterative
method, the \emph{inner iteration}.
This is for example the case when many eigenvectors are to be deflated,
or, more generally when the deflation subspace is large 
(and represented by a basis of sparse vectors thus making the use of an inner
iteration more attractive, see~\cite{FVConstruction2001,LuesLocal2007}).
In this context, it has been observed in \cite{NVComparison2006} that deflation
methods are quite sensitive to the accuracy of the inner iteration.

Given a stopping criterion for the outer iteration, i.e.,
\begin{equation}
    \label{eq:stopping-outer-iter}
    \norm{ r_i } \le \tau\, \norm{ b } =: \varepsilon
\end{equation}
for some $0 < \tau \ll 1$,
we now want to specify a stopping criterion for the inner iteration
which is of the form
\begin{equation}
    \label{eq:inner-stopping-criterion}
    \norm{ r_i^\textrm{c} } \le \tau^\textrm{c}\, \norm{ b^\textrm{c} }
    \,.
\end{equation}
The experiments in \cite{NVComparison2006} suggest that it is sufficient
to set a \emph{fixed tolerance}
\begin{equation}
    \label{eq:fixed-tolerance}
    \tau^\textrm{c} = \varepsilon \cdot c  \quad \text{with} \quad 0 < c \le 1.
\end{equation}
In our experiments we observed that this accuracy requirement is only crucial
for the first iterations and can be relaxed later on. This can be explained by
the theory for inexact Krylov subspace methods from
\cite{SSTheory2003,ESInexact2004} as follows.

We first observe that within the deflated CG method the matrix $A (I - \pi_A(\mathcal{S}))$ is only
used to compute matrix vector products, requiring the solution of the inner system \eqref{eq:inner-system}.
Due to the fact that the deflated matrix can be written as
\[
    A (I - \pi_A(\mathcal{S}))
    = A - A V (V^* A V)^{-1} V^* A
    \,.
\]
We can interprete the inexact calculation of %compute the matrix vector product with the matrix
$(V^* A V)^{-1}$ as the inexact matrix vector product by replacing %inexactly. Or, more generally, we replace
\[
    (V^* A V)^{-1}
    \quad\text{by}\quad
    (V^* A V)^{-1} + \Delta_i
\]
for some perturbation matrix $\Delta_i$---different in every iteration. Then
\begin{align*}
        A - A V ((V^* A V)^{-1} + \Delta_i) V^* A)
    &=
        A - A V (V^* A V)^{-1} V^* A -
        \underbrace{A V \Delta_i V^* A}_{=: E_i}
    \\
    &=
        A(I - \pi_A(\mathcal{S})) - E_i
    \,.
\end{align*}
Hence, the inaccurate solution of \eqref{eq:inner-system} is essentially equivalent
to an inaccurately computed matrix vector product.
(For a more general treatment of the perturbation of projections see
\cite{SteNumerical2011}.)

We start our discussion by considering the full orthogonalization method (FOM)
\cite{SaaIterative2003} which is equivalent to CG in exact arithmetic. (The
first vector should be computed with full accuracy.)
Assume we run FOM where the $i$th matrix vector
product is replaced by a product with the matrix
$A(I - \pi_A(\mathcal{S})) - E_i$. According to~\cite{SSTheory2003} the method converges to the desired
tolerance $\varepsilon$ if
\[
    \| E_i \| < C \tfrac{1}{\| r_i \|} \varepsilon
\]
for some constant $C > 0$ %  \cite{}
(see also \cite{ESInexact2004}) and it has been observed in~\cite{BFInexact2005} that the rate of convergence does not
change by much even when choosing $C = 1$.
%Choosing $C = 1$ often works well \cite{BFInexact2005}.

In particular we obtain due to $\| E_i \| \le \| A V \|^2 \cdot \| \Delta_i \|$ that 
the norm of $E_i$ is bounded by a factor proportional to the norm of $\Delta_i$. In turn the norm of
$\Delta_i$ can be bounded by a value that is proportional to the
error in the computation of $(V^* A V)^{-1} b^c$ and this error can be bounded by
a value proportional to the norm of the residual. Hence there exists a
constant $c > 0$ such that
\[
    \label{eq:adaptive-tolerance}
    \tau^\mathrm{c} = c \tfrac{1}{\| r_i \|} \varepsilon
\]
is sufficient for FOM to converge to the desired tolerance.
We will refer to this as the \emph{adaptive tolerance}, where
$r_i$ is the residual of the outer iteration.

Equation \eqref{eq:adaptive-tolerance} means
that the relative tolerance for the inner iteration can be relaxed 
while the outer iteration advances.

In exact arithmetic FOM is equivalent to the CG method. When using inexact CG
the loss of orthogonality can become a problem. Thus additional
orthogonalization might be required as in the flexible CG method
\cite{NotFlexible2000}.

%-------------------------------------------------------------------------
\section{Numerical Experiments}
\label{sec:experiments}

This section contains numerical experiments to illustrate the developed
theory and the quality of our convergence estimates.
We use the following three test matrices called
\emph{Simple}, {\em Poisson} and {\em NOS1}.
\begin{description}
    \item[Simple]
        The {\em Simple} matrix is the
        diagonal matrix
        $A = \diag(10^{-2}, 1, \dots, 1) \in \Rmvs{100}{100}$
        with eigenvalues $\lambda_{100} = 10^{-2}$ and $\lambda_1 = \dots =
        \lambda_{99} = 1$.

        The convergence rate of the (deflated) CG method is invariant under
        simultaneous unitary transformations of the matrix, right hand side,
        and initial guess. Thus without loss of generality we can consider
        diagonal matrices.

        The condition number of the Simple matrix is 100. If the deflation
        subspace is spanned by the eigenvector corresponding to the
        smallest eigenvalue $e_1$ then the condition number of the deflated
        matrix is $1$. Any deflation subspace that is orthogonal to $e_1$
        yields an effective condition number of $100$.

        If we pick a perturbation orthogonal to $e_1$ and increase its size
        step by step, the condition number should gradually go from $1$ to
        $100$. This will give us a first impression of the behavior of the
        bounds.
    \item[Poisson]
        The {\em Poisson$(N)$} matrix is the $N^2 \times N^2$ matrix
        arising from the finite element discretization of Poisson's
        equation using quadratic bilinear elements on a uniform $N \times N$
        grid. It is one of the simplest, non-trivial matrices
        that appear in real world applications.
    \item[NOS1] The {\em NOS1} matrix is the $237 \times 237$ 
        matrix from \cite{MatrixMarket}.
        According to its description, it is a 
        ``finite element approximation to [the] biharmonic 
        operator on a beam with one end
        free and one end fixed.''
        We choose this as our third example as it has a large
        condition number of about $2 \cdot 10^{7}$.
\end{description}
From the different mathematical equivalent formulations of the deflated CG
method
\cite{GGLNFramework2013} we chose the one from
Saad, Yeung, Erhel and Guyomarc'h \cite{SYEGDeflated2000}.
Throughout this section we refer to the \emph{optimal condition number} 
as the \emph{effective condition number} of the deflated matrix for the
unperturbed subspace $\mathcal{S}$. 
We refer to the \emph{original condition number} 
as the condition number of the matrix $A$.
We refer to the \emph{effective condition number}
as the effective condition number of the deflated matrix for the
perturbed subspace $\widetilde{\mathcal{S}}$.
Analogously the optimal, original and effective smallest eigenvalues is the
smallest non-zero eigenvalue of $A(I-\pi_A(\mathcal{S})$, $A$ and
$A(I-\pi_A(\widetilde{\mathcal{S}}))$, respectively.

\subsection{Perturbation in Eigenvector Deflation}\label{sec:pertubation}

In this section we illustrate the perturbation theory
from
Section~\ref{sec:inexact-deflation}.
We denote the eigenvalues of the matrix $A$ under consideration
by $\lambda_1 \ge \lambda_2 \ge \dots \ge \lambda_n \ge 0$ 
with corresponding eigenvectors $q_1, \ldots, q_n$.

\begin{figure}
    \adjustbox{width=0.49\textwidth}{\input{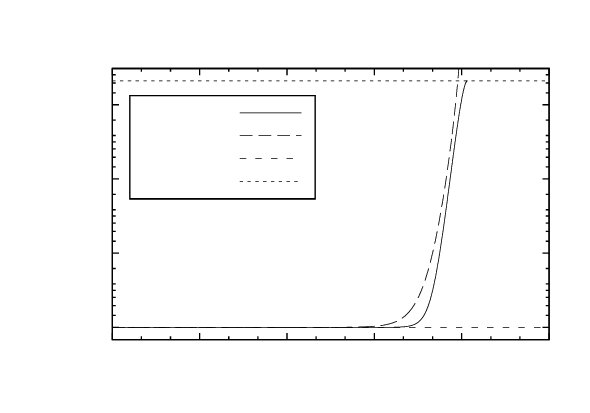}}
    \hfill
    \adjustbox{width=0.49\textwidth}{\input{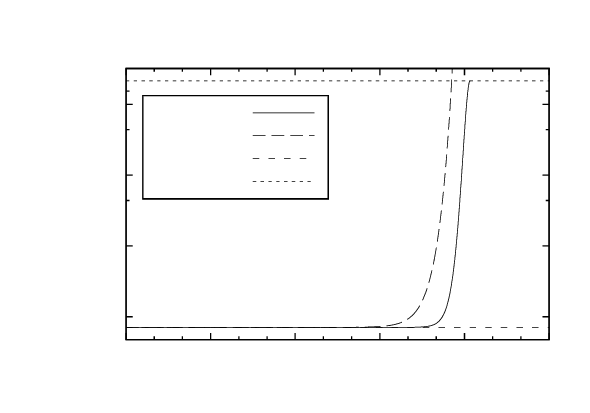}}

    \caption{Estimated and computed effective condition number for the Simple
    (left) and Poisson($31$) (right) matrix as a function of the principal
    angle $\theta$.}
    \label{fig:est-simple-poisson-k1}
\end{figure}

First, we consider the {\em Simple} matrix and choose
$\mathcal{S} = \Span\{ q_n \}$.
We choose the perturbation vector $E = \alpha q_{n-1}$, $\alpha \in
\mathbb{R}$,
thus $\tilde{\mathcal{S}} = \Span\{ q_n + \alpha q_{n-1} \}$ is the
perturbed deflation space.
This choice of perturbation has a large impact on the effective condition number.
Thus it will illustrate the sharpness of the developed bound.
We compute the effective condition number of the matrix
$A (I - \pi_A(\tilde{\mathcal{S}}))$ and the estimate from
Section~\ref{sec:inexact-deflation}.
The results are given in the left part of Figure~\ref{fig:est-simple-poisson-k1}.
Here the graph ``effective'' shows the effective condition number, the
graph ``eff.\ est.''\ shows the value from \eqref{eq:kappa_eff_perturbed}.
We note that the estimate recovers the behavior of the actual effective condition
number quite well.

\begin{figure}
    \adjustbox{width=0.49\textwidth}{\input{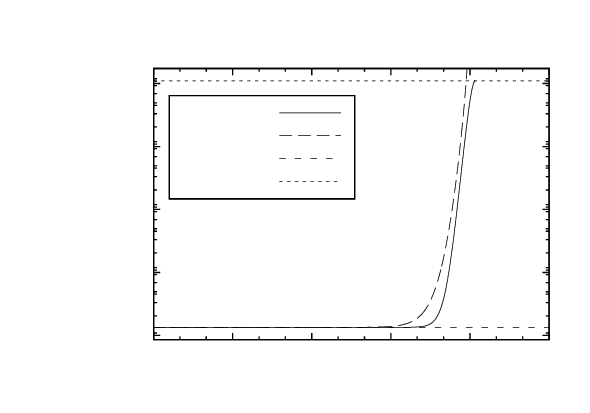}}
    \hfill
    \adjustbox{width=0.49\textwidth}{\input{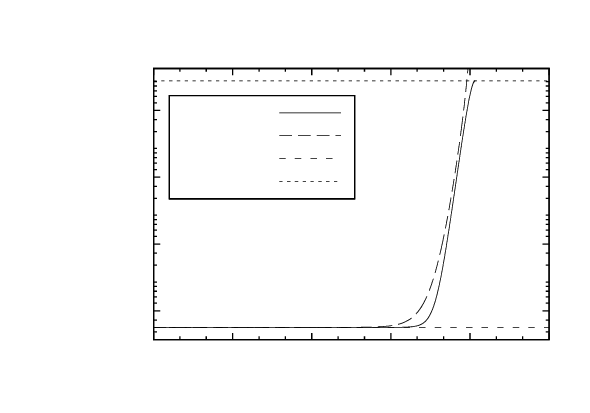}}

    \caption{Estimated and computed effective condition for the
    Poisson($31$) matrix, deflating $100$~eigenvectors, (left) and NOS1 matrix, 
    deflating $25$ eigenvectors, (right)
    as a function of the principal angle $\theta$.}
    \label{fig:est-poisson31-nos1-10-percent}
\end{figure}

We run the same test for the Poisson($31$) matrix
and report the results in the right
part of Figure~\ref{fig:est-simple-poisson-k1}.
Qualitatively the results are very similar to those for the matrix Simple
with the difference between the estimate and the effective
condition number being slightly larger. 
Deflating the
eigenvectors corresponding to the $k$ smallest eigenvalues, i.e.,
the subspaces $\mathcal{S} = \Span\{ q_{n - (k - 1)}, \dots, q_n \}$
and 
$E = \alpha \left[ q_{n - 2(k - 1)} | \dots | q_{n - (k-1) - 1} \right]$
yields similar results as shown in
Figure~\ref{fig:est-poisson31-nos1-10-percent};
we choose a deflation subspace consisting of approximately $n/10$
eigenvectors corresponding to the smallest eigenvalues of the
Poisson$(31)$ and the NOS$1$ matrix, respectively.

\begin{figure}
    \begin{center}
    \adjustbox{width=0.49\textwidth}{\input{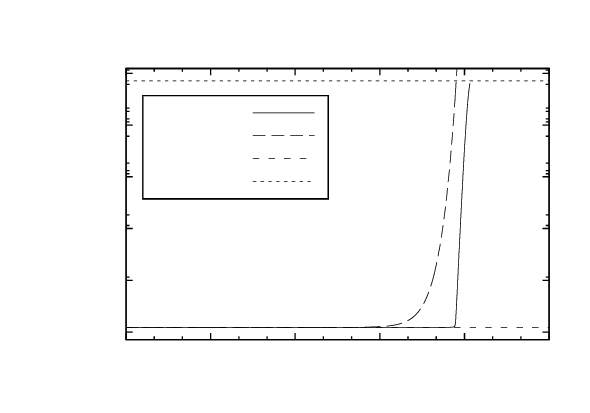}}
    \end{center}
    \caption{Estimated and computed effective condition for the
        Poisson($31$) matrix, $4$ eigenvectors and a random perturbation
        as a function of the principal angle $\theta$.}
    \label{fig:est-poisson31-k4-randn}
\end{figure}

The perturbations chosen so far had a relatively large effect on the effective
condition number.  To illustrate that we conduct a test with a random
perturbation matrix $E = \alpha R$. In here the entries of $R$ are normally
$\mathcal{N}(0,1)$ distributed random numbers. In
Figure~\ref{fig:est-poisson31-k4-randn} we can see that the theoretical bound
is not as sharp as in the previous examples, but the qualitative behavior is
again captured.

\begin{figure}
	\begin{center}
		\adjustbox{width=0.49\textwidth}{\input{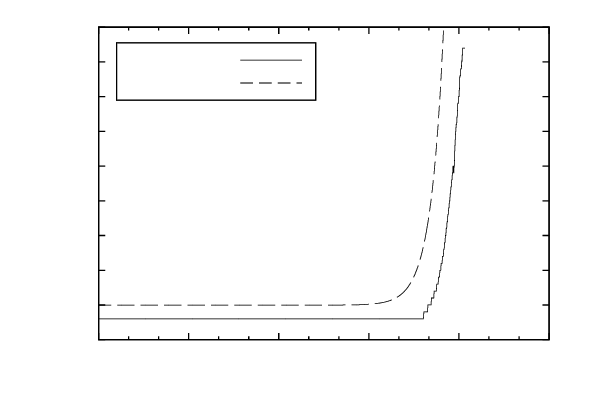}}
	\end{center}
	\caption{Number of iterations of the deflated CG method for the
        Poisson$(31)$ matrix in dependence of the size of the perturbation.
        Deflation of the $50$ smallest eigenvectors.
        Reduction of the norm of the residual by a factor of $10^{-6}$.
        }\label{fig:iterationNr-aggregation}
\end{figure}

Based on the convergence bound~\eqref{eq:convergence_estimate_dCG} of deflated
CG one can easily derive an estimate for the number of iterations deflated CG
requires to reduce the $\ell_2$-norm of the residual to a prescribed tolerance
$\tau$ in
terms of $\kappa_\mathrm{eff}$. To this extend first observe that due to 
\eqref{eq:convergence_estimate_dCG}
\begin{equation*}
	\| r_i \|_2
	\le 2  \sqrt{\kappa} \left( \frac{ \sqrt{\kappa_\eff} - 1 }{ \sqrt{\kappa_\eff} + 1} \right)^i
    \| r_0 \|_2
	\quad \text{for } i = 0, 1, 2, \ldots
\end{equation*}
and thus
\[
    i \ge \frac{\log(\tau / (2 \sqrt{\kappa}))}{
        \log ( (\sqrt{\kappa_\eff} - 1 ) / ( \sqrt{\kappa_\eff} + 1) ) }
\]
to reduce the $\ell_2$-norm of the residual by a factor of $\tau$.

In order to illustrate that our bound for
$\kappa_{\mathrm{eff}}$ can be used as a tool to set the required accuracy of
eigenvectors used in the deflation subspace, we report in
Figure~\ref{fig:iterationNr-aggregation} the number of iterations deflated CG
required to converge as a function of the perturbation of the smallest $k=50$
eigenvectors. As one can see the estimate gives a fairly good idea about the
accuracy requirement, i.e., it correctly captures the steep increase of the
number of iteration. We stop the iteration when the norm of the residual is
reduced by a factor of $10^{6}$.

\begin{figure}
    \begin{center}
        \adjustbox{width=0.49\textwidth}{\input{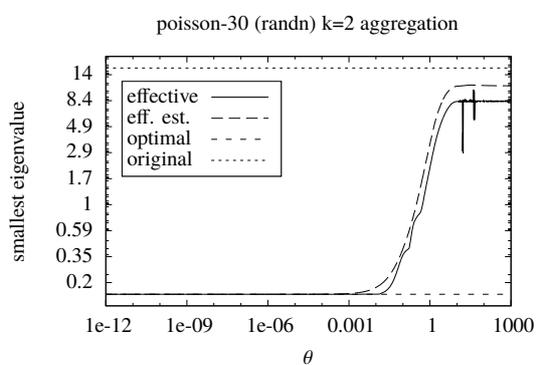}}
    \end{center}
    \caption{Aggregation with $k=2$ for Poisson$(30)$, random
    perturbation. Estimated and measured $\mu_k^{-1}$.}
    \label{fig:poisson-aggregation}
\end{figure}

As a final experiment we consider a deflation subspace that is not directly
spanned by eigenvectors.
We consider the Poisson($30$) matrix and construct $V$ as described in 
\cite{LuesLocal2007}.
That is, we take eigenvectors to the smallest two eigenvalues and ``chop''
them up over $2\times 2$ aggregates, i.e., the $30\times30$ grid is divided
into $2\times 2$ blocks and the $15^2\cdot2$ columns of $V$ are simply the
entries of the eigenvectors restricted to these blocks. 
The perturbed basis $\widetilde{V}$ is obtained by repeating the same
procedure with the same eigenvectors that are perturbed by two
scaled $\mathcal{N}(0,1)$ random vectors.
In Figure~\ref{fig:poisson-aggregation} we show
the estimate for $\mu_{k}^{-1}$ together with the actual $\mu_{k}^{-1}$ in
dependence of the largest principal angle between the subspace spanned by the
columns of $V$ and $\widetilde{V}$.
Again we
see that the qualitative behavior is captured well by the estimate, but as
expected the gap between estimate and actual value of $\mu_{k}^{-1}$ is larger
than in the cases where $V$ is built of orthonormal vectors.

\begin{figure}
    \begin{center}
        \adjustbox{width=0.49\textwidth}{\input{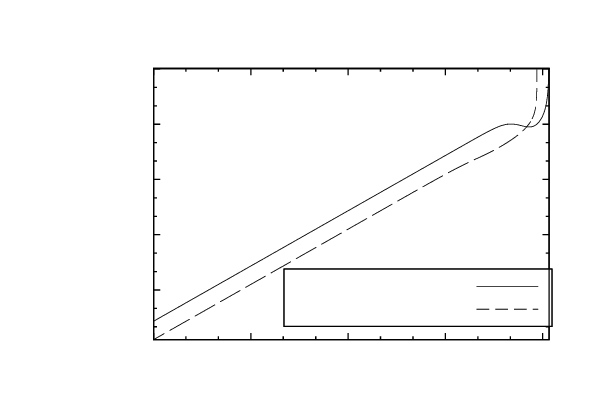}}
        \hfill
        \adjustbox{width=0.49\textwidth}{\input{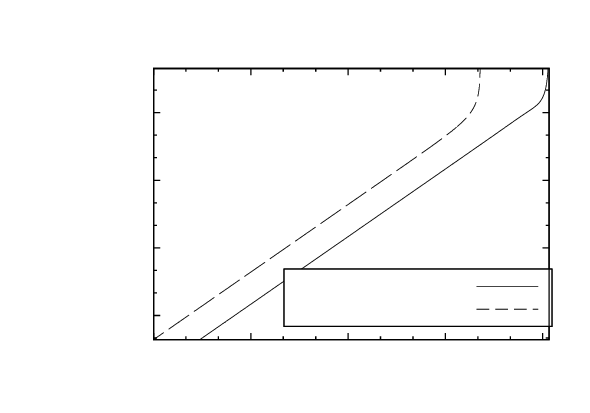}}
    \end{center}
    \caption{The relative error of the eigenvalue bound from
    \protect\cite[Lemma~3.30]{GauRecycling2014} and
    Proposition~\ref{prop:eigenvalue-perturbation}
    for the Simple (left) and Poisson(31) (right) matrix.}
    \label{fig:quality-of-bounds}
\end{figure}

In \cite[Lemma~3.30]{GauRecycling2014} a bound for the difference of the eigenvalues of the deflated matrix in the
perturbed and in the unperturbed case is given. If the effective
condition number of the unperturbed deflated matrix is known then a bound for
the effective condition number in the perturbed case is easily derived. We
compare this bound with the result from
Proposition~\ref{prop:eigenvalue-perturbation}.
To this extend we compare the relative error, i.e.,
\[
    \frac{| \text{bound} - \kappa_\eff |}{\kappa_\eff}
    \,.
\]
of the two bounds in Figure~\ref{fig:quality-of-bounds}.
The bound of \cite{GauRecycling2014} is better than the bound from
Proposition~\ref{prop:eigenvalue-perturbation} for the Simple matrix. For the
Poisson$(31)$ matrix the situation is the other way around.

\subsection{Accuracy of the Inner System}

\begin{figure}
    \adjustbox{width=0.49\textwidth}{\input{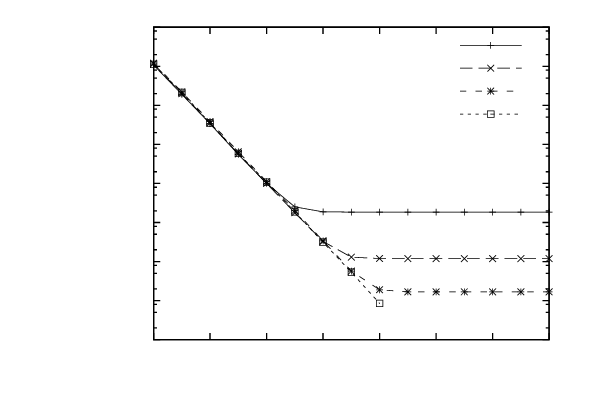}}
    \hfill
    \adjustbox{width=0.49\textwidth}{\input{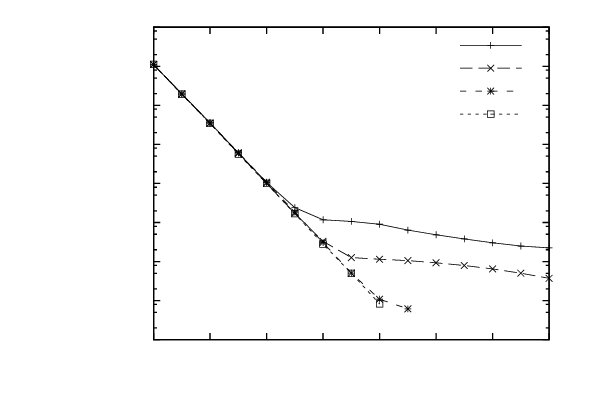}}
    \caption{Deflation of $200$ eigenvectors of the Poisson($31$) matrix
    with fixed (left) and adaptive (right) inner tolerance.}
    \label{fig:different-abs-tolerance}
\end{figure}

We now consider the Poisson($31$) matrix and choose
the columns of $V$ as the $200$ eigenvectors corresponding to the smallest
eigenvalues of the Poisson($31$) matrix.
In Figure~\ref{fig:different-abs-tolerance} we show the
convergence of the deflated CG iteration for varying inner
accuracy requirements, i.e., the accuracy of the solution of
$V^* A Ve = V^* A r$ approximated by an inner CG iteration. 
The left hand side of the figure shows the norm of the residual for
different choices of $\tau^c$ in \eqref{eq:inner-stopping-criterion};
the right hand side shows these norms for different choices of 
$\tau^c \norm{r_i}$; see \eqref{eq:adaptive-tolerance}.

From the discussion in section~\ref{sec:accuracy-discussion} we expect that
the adaptive tolerance \eqref{eq:adaptive-tolerance} will produce similar
results to the fixed one \eqref{eq:fixed-tolerance} and this can indeed
be observed in
Figure~\ref{fig:different-abs-tolerance}.
Choosing the tolerance adaptively produces even slightly better results.

In numerical experiments for other choices for the deflation subspace $V$
which we do not report here, we observed no qualitative difference of the two
choices for the inner tolerance criterion.

\begin{table}
    \caption{Number of inner iterations for both tolerance selection
    strategies ($\tau = 10^{-6}$, $c = 0.1$).
    }
    \begin{center}
    \begin{tabular}{@{}llrr@{}}
        \toprule
        It. & 
        \multicolumn{1}{c}{$\norm{r_i}$} &
        \multicolumn{2}{c}{Inner-It.} \\
        \cmidrule(l){3-4}
            &                    & Fixed & Adaptive \\
        \midrule
        0   & \scinum{8.21}{0}   & 35    & 35 \\
        1   & \scinum{1.15}{0}   & 27    & 27 \\
        2   & \scinum{2.15}{-1}  & 27    & 24 \\
        3   & \scinum{3.93}{-2}  & 27    & 20 \\
        4   & \scinum{6.41}{-3}  & 27    & 16 \\
        5   & \scinum{1.13}{-3}  & 27    & 11 \\
        6   & \scinum{1.96}{-4}  & 27    &  8 \\ 
        7   & \scinum{3.25}{-5}  & 27    &  4 \\
        8   & \scinum{5.70}{-6}  & 27    &  2 \\
        9   & \scinum{9.15}{-7}  & 27    &  0 \\
        \bottomrule
    \end{tabular}
    \end{center}

    \label{tbl:inner-it}
\end{table}

Motivated by this observation 
we run the same test again
using the adaptive tolerance
\eqref{eq:adaptive-tolerance}
as the stopping criterion
and monitor the number of inner iterations in each outer iteration.
We choose
$c = 0.1$
and solve to a tolerance of 
$\tau = 10^{-6}$.
As the results in Table~\ref{tbl:inner-it} show, less and less iterations of
the inner CG method are needed when the outer iteration advances and the
adaptive tolerance is used.  For the fixed tolerance a high number of inner
iterations is required throughout the whole outer iteration.  We note that
the convergence is not distinguishable from the case where an exact inner
solve would be performed. 

\section{Conclusions}

We have shown that to get an optimal convergence rate of the deflated CG
method the eigenvectors only have to be computed to a limited accuracy.  This
is particularly interesting in situations where a linear system has to be
solved for many right hand sides and the additional overhead of computing
approximations, e.g.\ by using ARPACK~\cite{LSYARPACK1998}, of the smallest
eigenvectors can be compensated.  This is for example the case when computing
rational approximations of matrix functions~\cite{FSMatrix2008a} or in the
computation of expectation values in statistical physics, e.g., Lattice Gauge
theory~\cite{GLBQuantum2010}.

\section*{Acknowledgments}

We would like to thank the referees for many helpful remarks and suggestions and
Andreas Frommer for his help and advice.
Furthermore, the work 
is partly supported by the German Research
Foundation (DFG) through the Priority Programme 1648 
``Software for Exascale Computing'' (SPPEXA)
and the Transregional Collaborative Research Centre 55 (SFB/TRR55)
``Ha\-dron Physics from Lattice QCD''.

\ifx\preprint\undefined
\bibliographystyle{siam}
\bibliography{defl_mg_theo}
\else
\printbibliography
\fi

\label{end-of-document}

\end{document}